\documentclass[review]{elsarticle}
\usepackage{lineno,hyperref,amsmath,amssymb,amscd,verbatim}
\newtheorem{Theorem}{Theorem}[section]
\newtheorem{Lemma}{Lemma}[section]
\newtheorem{Question}{Question}[section]
\numberwithin{equation}{section}
\newcommand{\supp}{\mathrm{supp}}
\newproof{proof}{Proof}
\journal{Journal of Differential Equations}

\begin{document}

\begin{frontmatter}

\title{
An explicit compact universal space for real flows
}

\author[mymainaddress]{Yonatan Gutman}
\ead{y.gutman@impan.pl}

\author[mymainaddress]{Lei Jin\corref{mycorrespondingauthor}}
\cortext[mycorrespondingauthor]{Corresponding author}
\ead{jinleim@mail.ustc.edu.cn}

\address[mymainaddress]{Institute of Mathematics, Polish Academy of Sciences, ul. \'Sniadeckich 8, 00-656 Warszawa, Poland}

\begin{abstract}
The Kakutani-Bebutov Theorem (1968) states that any compact
metric real flow whose fixed point set is
homeomorphic to a subset of $\mathbb{R}$ embeds into
the Bebutov flow,
the $\mathbb{R}$-shift on $C(\mathbb{R},[0,1])$. An interesting fact is that
this universal space is a function space.
However, it is not compact, nor locally compact.
We construct an explicit countable product of compact subspaces of
the Bebutov flow
which is a universal space for all compact metric real flows,
with no restriction;
namely,
into which any compact metric real flow embeds.
The result is compared to previously known universal spaces.
\end{abstract}

\begin{keyword}
Bebutov-Kakutani theorem\sep
compact metric flow\sep
Bebutov flow\sep
Bernstein space\sep
universal space\sep
compactification\sep
equivariant embedding

\MSC[2010] 37B05\sep 54H20
\end{keyword}

\end{frontmatter}


\section{Introduction}

A tuple $(X,G)$ is called a {\it topological dynamical system} if
$X$ is a topological space, $G$ is a topological group with the unit $e_G$, and
$\Gamma:G\times X\rightarrow X,(g,x)\mapsto gx$ is a continuous mapping
satisfying $\Gamma(e_G,x)=x$ and $\Gamma(g_1,\Gamma(g_2,x))=\Gamma(g_1g_2,x)$
for each $g_1,g_2\in G$ and $x\in X$.
The most studied topological dynamical systems are those where
$G=\mathbb{Z}$ or $G=\mathbb{R}$.
For the case $G=\mathbb{R}$
we say that $(X,\mathbb{R})$ is a {\it real flow} (or a {\it flow}, for short).
In this paper, we concentrate on real flows.
A flow $(X,\mathbb{R})$ is said to be a compact metric flow (resp. a separable metric flow, etc.)
if $X$ is a compact metric space (resp. a separable metric space, etc.).
We will mainly focus on compact metric flows.

There are many examples of real flows.
In fact,
there are natural constructions for passing from a $\mathbb{Z}$-action to a flow,
and vice versa \cite[Section 1.11]{BS}.
Let $(X,T)$ be a $\mathbb{Z}$-action, meaning that $X$ is a compact metric space and $T$ is a
homeomorphism of $X$. Let $f:X\rightarrow\mathbb{R}^+$ be a continuous function
bounded away from $0$.
Consider the quotient space $$S_fX=\{(x,t)\in X\times\mathbb{R}^+:0\le t\le f(x)\}/\sim,$$
where $\sim$ is the equivalence relation $(x,f(x))\sim(Tx,0)$.
The {\it suspension} (flow) over $(X,T)$ generated by the {\it roof function} $f$ is
the flow $(S_fX,(\psi_t)_{t\in\mathbb{R}})$ given by
$$\psi_t(x,s)=(T^nx,s') \text{ for } t\in\mathbb{R} \text{ and } (x,s)\in S_fX,$$
where $n$ and $s'$ satisfy
$$\sum_{i=0}^{n-1}f(T^ix)+s'=t+s, \,\; 0\le s'\le f(T^nx).$$
In other words, flow along $\{x\}\times\mathbb{R}^+$ to $(x,f(x))$ then continue from
$(Tx,0)$ (which is the same as $(x,f(x))$) along $\{Tx\}\times\mathbb{R}^+$ and so on.
Since $X$ is a compact metric space, the suspension is also a compact metric flow.

\medskip

It is natural to seek one compact metric flow whose subflows exhaust all compact metric flows up to isomorphism.
Moreover, the simpler this flow is the better.
Formally, let $(X, \mathbb{R})=(X, (\varphi_r)_{r\in\mathbb{R}})$
and $(Y, \mathbb{R})=(Y, (\phi_r)_{r\in\mathbb{R}})$ be two flows.
We say that $(Y,\mathbb{R})$ {\it embeds} into $(X,\mathbb{R})$ if there exists
an $\mathbb{R}$-equivariant homeomorphism of $Y$ onto a subspace of $X$; that is,
there is a homeomorphism $f: Y\rightarrow f(Y)\subset X$ such that
$f\circ \phi_r=\varphi_r\circ f$ for all $r\in\mathbb{R}$.
Such a map $f$ is called an {\it embedding} of the flow $(Y,\mathbb{R})$ into the flow $(X,\mathbb{R})$.
We note that in this definition $f(Y)$ is equipped with the subspace topology.
Let $\mathcal{C}$
be a family of flows.
The flow $(X,\mathbb{R})$ is said to be a {\it universal flow} (or sometimes said to be a
{\it universal space}) for $\mathcal{C}$ if all
$(Y,\mathbb{R})\in\mathcal{C}$ embed into $(X,\mathbb{R})$.
The more canonical the universal space is the more useful is its universality.

Natural candidates for universal flows are function spaces.
Let $C(\mathbb{R},[0,1])$ be the metric space of all continuous functions $f:\mathbb{R}\rightarrow[0,1]$
equipped with the topology of uniform convergence on compact sets.
Let $\mathbb{R}=(\varphi_r)_{r\in\mathbb{R}}$ act on $C(\mathbb{R},[0,1])$ as follows (the $\mathbb{R}$-shift):
$(\varphi_rf)(x)=f(x+r)$ for each $x,r\in\mathbb{R}$ and $f\in C(\mathbb{R},[0,1])$.
We denote this flow by $(C(\mathbb{R},[0,1]),\mathbb{R})$, the $\mathbb{R}$-shift on $C(\mathbb{R},[0,1])$.
\footnote{If one tries to endow $C(\mathbb{R},[0,1])$ with the topology induced by
the metric $d(f,g)=\sup_{x\in\mathbb{R}}|f(x)-g(x)|$ where $f,g\in C(\mathbb{R},[0,1])$,
then $(C(\mathbb{R},[0,1]),\mathbb{R})$ is no longer a topological dynamical system.}
The flow $(C(\mathbb{R},[0,1]),\mathbb{R})$ is called the Bebutov flow \cite[Chapter 13]{A}.
The terminology differs from the one appearing in very early articles \cite{N,Ko}.
In the present paper, we follow Auslander's terminology \cite[Chapter 13]{A}.
The following is the famous Kakutani-Bebutov theorem.

\begin{Theorem}\label{KB}
Let $(X,\mathbb{R})=(X,(\varphi_r)_{r\in\mathbb{R}})$ be a compact metric flow
and $F=\{x\in X:\varphi_rx=x \text{ for all } r\in\mathbb{R}\}$ be the fixed point set of the flow.
If $F$ is homeomorphic to a subset of $\mathbb{R}$, then $(X,\mathbb{R})$ embeds into
the Bebutov flow $(C(\mathbb{R},[0,1]),\mathbb{R})$.
\end{Theorem}

\medskip

The embeddability of a compact metric flow in $(C(\mathbb{R},[0,1]),\mathbb{R})$
was first proven by Bebutov (see \cite{B,N}) under the assumption that
the fixed point set has at most two points. The general version was proven by Kakutani \cite{K} in 1968
(see also \cite[Chapter 13]{A}).
In particular,
any compact metric flow with no fixed points
embeds into $(C(\mathbb{R},[0,1]),\mathbb{R})$.
This is a wonderful theorem,
\footnote{The condition appearing in the theorem is necessary as
the fixed point set of $(C(\mathbb{R},[0,1]),\mathbb{R})$ consists of all constants in $[0,1]$.
Thus,
$(X,\mathbb{R})$ embeds into $(C(\mathbb{R},[0,1]),\mathbb{R})$
if and only if
$F$ (topologically) embeds into $\mathbb{R}$.}
and it is very satisfying that $C(\mathbb{R},[0,1])$ of
the universal space is a function space.
However, we remark that
in this theorem, the compact metric flow $(X,\mathbb{R})$ needs to satisfy some additional conditions;
meanwhile, the universal space $(C(\mathbb{R},[0,1]),\mathbb{R})$
is not compact, nor locally compact.
This is unfortunate.
The first is unfortunate because the theorem does not encompass all compact metric real flows.
The second is unfortunate as one would like the universal space to belong to
the family for which it is universal;
namely, we would like the universal space to be a compact metric real flow.

\medskip

Our aim is to find a compact metric universal flow for all compact metric flows;
at the same time, as {\it simple} or {\it canonical} as possible.
Indeed, we also would like our universal flow to be as ``close'' to the Bebutov flow $(C(\mathbb{R},[0,1]),\mathbb{R})$
as possible. However, trivially the Bebutov flow is not universal, nor is any finite product of it.
\footnote{The fixed point set of $(C(\mathbb{R},[0,1]^N),\mathbb{R})$ (where $N$ is a natural number)
is $[0,1]^N$ into which not every compact metric space embeds.}
So we turn to considering the countable product $(C(\mathbb{R},[0,1]^\mathbb{N}),\mathbb{R})$.
More precisely, let $C(\mathbb{R},[0,1]^\mathbb{N})$ be the metric space of all continuous functions
$f:\mathbb{R}\rightarrow[0,1]^\mathbb{N}$ endowed with the metric
$$d(f,g)=\sum_{n\in\mathbb{N}}\sum_{N\in\mathbb{N}}\frac{||f_n-g_n||_{L^\infty([-N,N])}}{2^{n+N}}$$
for every $f=(f_n)_{n\in\mathbb{N}},g=(g_n)_{n\in\mathbb{N}}\in C(\mathbb{R},[0,1]^\mathbb{N})$,
and let $\mathbb{R}=(\varphi_r)_{r\in\mathbb{R}}$ act on $C(\mathbb{R},[0,1]^\mathbb{N})$ as follows:
for any $r\in\mathbb{R}$ and any $f\in C(\mathbb{R},[0,1]^\mathbb{N})$,
$(\varphi_rf)(x)=f(x+r)$ for all $x\in\mathbb{R}$.
It is well known
that all compact metric spaces embed into the compact metric space $[0,1]^\mathbb{N}$.
This can be used to show in a straightforward manner that
$(C(\mathbb{R},[0,1]^\mathbb{N}),\mathbb{R})$ is a universal flow
for all compact metric flows (for exact details see Section 2).
Note that this flow is a complete separable metric flow; however,
it is not compact nor locally compact.

In this paper, we find an explicit countable product of compact subspaces of
the Bebutov flow $(C(\mathbb{R},[0,1]),\mathbb{R})$
which is a universal space for all compact metric flows,
with no restriction.
The key idea for the solution is to find a countable family of
compact function spaces which allow to ``separate points'' equivariantly
for each compact metric flow.
These spaces surprisingly appeared in a related problem for $\mathbb{Z}$-actions.

For $\mathbb{Z}$-actions $(X,T)$
one can similarly consider the embeddability for $(X,T)$.
It is easy to see that the shift on
$([0,1]^\mathbb{N})^\mathbb{Z}$ is universal.
But an important and much more difficult question is:
for which spaces the shift on $([0,1]^N)^\mathbb{Z}$
(where $N\in\mathbb{N}$) is universal?
Here we briefly state some related results.
Under the conditions that the space $X$ is finite dimensional and
$(X,T)$ has no periodic points, Jaworski \cite{J} showed in 1974 that
$(X,T)$ embeds into the shift on $[0,1]^{\mathbb{Z}}$.
Then the first author \cite{G} extended this result to the case of finite dimensional
systems having reasonable amount of periodic points.
We mention that the embedding problem is very close to the theory of mean dimension
which was introduced by Gromov \cite{Gr} and systematically developed by
Lindenstrauss and Weiss \cite{L,LW}.
Recently, Tsukamoto and the first author \cite{GT} proved, using Fourier and complex analysis,
that any minimal system of mean dimension strictly less than $N/2$
embeds into the shift on $([0,1]^N)^\mathbb{Z}$.
We remark that the value $N/2$ is optimal because
Lindenstrauss and Tsukamoto \cite[Theorem 1.3]{LT} constructed a minimal system of mean dimension
$N/2$ which never embeds into the shift on $([0,1]^N)^\mathbb{Z}$.

\medskip

We now present the main result of our paper.
Before we state our main result,
we recall some necessary notions and results.
We will use the theory of harmonic analysis.
First recall that a {\it rapidly decreasing function} $f$ is an infinitely differentiable function
on $\mathbb{R}$ satisfying $$\lim_{|x|\rightarrow\infty}x^nf^{(j)}(x)=0$$
for all $n,j\in\mathbb{N}$.
Denote by $S(\mathbb{R})$ the space of all rapidly decreasing functions
equipped with the topology given by the family of seminorms
$$||f||_{j,n}=\sup|x^nf^{(j)}(x)|.$$
This topology on $S(\mathbb{R})$ is metrizable and $S(\mathbb{R})$
is complete; namely, $S(\mathbb{R})$ becomes a Fr\'{e}chet space.
We call a continuous linear functional on $S(\mathbb{R})$
a {\it tempered distribution} on $\mathbb{R}$
(for details see \cite[Chapter VI.4]{Ka}).

For $L^{1}$-functions (in particular, rapidly decreasing functions)
$f:\mathbb{R}\rightarrow\mathbb{C}$, the definition of the Fourier
transforms is given by
\[
\mathcal{F}(f)(\xi)=\int_{-\infty}^{\infty}e^{-2\pi\sqrt{-1}t\xi}f(t)dt,\,\,\,\overline{\mathcal{F}}(f)(t)=\int_{-\infty}^{\infty}e^{2\pi\sqrt{-1}t\xi}f(\xi)d\xi.
\]
We have $\overline{\mathcal{F}}(\mathcal{F}(f))=f$
when $\mathcal{F}(f)\in L^1(\mathbb{R})$
and $\mathcal{F}(\overline{\mathcal{F}}(f))=f$
when $\overline{\mathcal{F}}(f)\in L^1(\mathbb{R})$.
For functions $f,g\in L^2(\mathbb{R})\cap L^1(\mathbb{R})$
we have by Plancherel's theorem \cite[Chapter VI.3.1, pp. 156]{Ka}
$\int f\overline{g}=\langle f,g\rangle=\langle\mathcal{F}(f),\mathcal{F}(g)\rangle$.
Assuming in addition
$\overline{\mathcal{F}}(f),\overline{\mathcal{F}}(g)\in L^1(\mathbb{R})$
and setting $f'=\overline{\mathcal{F}}(f)$ and $g'=\overline{\mathcal{F}}(g)$,
we have $\langle f,g\rangle=\langle\overline{\mathcal{F}}(f),\overline{\mathcal{F}}(g)\rangle$.
In particular letting $f=\mathcal{F}(\psi)$ and $g=\varphi$, we have
$\langle\mathcal{F}(\psi),\varphi\rangle=\langle\psi,\overline{\mathcal{F}}(\varphi)\rangle$.

We extend $\mathcal{F}$ and $\overline{\mathcal{F}}$ to tempered
distributions (in particular, to bounded continuous functions) $\psi$
by the dualities $\langle\mathcal{F}(\psi),\varphi\rangle=\langle\psi,\overline{\mathcal{F}}(\varphi)\rangle$
and $\langle\overline{\mathcal{F}}(\psi),\varphi\rangle=\langle\psi,\mathcal{F}(\varphi)\rangle$
where $\varphi$ are rapidly decreasing functions. For example, $\mathcal{F}(1)=\boldsymbol{\delta}_{0}$
is the delta probability measure at the origin. Recall that $\mathrm{supp}\mathcal{F}(f)\subset[a,b]$
means that the pairing $\langle\mathcal{F}(f),g\rangle$ vanishes
for any rapidly decreasing function $g:\mathbb{R}\to\mathbb{C}$ with
$\mathrm{supp}(g)\cap[a,b]=\emptyset$.

Let $a<b$ be two real numbers. We define $B_{1}^{\mathbb{C}}(V[a,b])$
as the space of all bounded continuous functions $f:\mathbb{R}\rightarrow\mathbb{C}$
satisfying $\supp\mathcal{F}(f)\subset[a,b]$ and $||f||_{L^{\infty}(\mathbb{R})}\le1$.
An important nontrivial fact is that $B_{1}^{\mathbb{C}}(V[a,b])$
is a compact metric space with respect to the distance $\boldsymbol{d}$
defined by
\[
\boldsymbol{d}(f_{1},f_{2})=\sum_{n=1}^{\infty}\frac{||f_{1}-f_{2}||_{L^{\infty}([-n,n])}}{2^{n}}
\]
where $f_{1},f_{2}\in B_{1}^{\mathbb{C}}(V[a,b])$, which coincides
with the standard topology of tempered distributions (see \cite[Lemma 2.3]{GT},
\cite[Chapter 7, Section 4]{Schwartz}).
One can also show that
$$B_{1}^{\mathbb{C}}(V[a,b])=\overline{\{f:\mathbb{R}\to\mathbb{C}|\,
f\textrm{ is rapidly decreasing,}\,
||f||_{L^{\infty}(\mathbb{R})}\le1,\;
\supp\mathcal{F}(f)\subset[a,b]\}},$$
where the closure is taken with respect to
the distance ${\boldsymbol{d}}$.

\medskip

We now define the flow $(B_1^\mathbb{C}(V[a,b]),\mathbb{R})=(B_1^\mathbb{C}(V[a,b]),(\tau_r)_{r\in\mathbb{R}})$ as follows:
for every $r\in\mathbb{R}$ and every $f\in B_1^\mathbb{C}(V[a,b])$,
we define $(\tau_rf)(t)=f(t+r)$ for each $t\in\mathbb{R}$.
By noting that $\supp\mathcal{F}(\tau_rf)=\supp\mathcal{F}(f)$
(see \cite{Ka})
we can easily check that
$(B_1^\mathbb{C}(V[a,b]),\mathbb{R})$ is indeed a flow.
Since $B_1^\mathbb{C}(V[a,b])$ is a compact metric space,
we obtain a compact metric flow
$(B_1^\mathbb{C}(V[a,b]),\mathbb{R})$,
the $\mathbb{R}$-shift on $B_1^\mathbb{C}(V[a,b])$.

Let $a>0$ and define $$B_1V_{-a}^a=\{f:\mathbb{R}\rightarrow[-1,1]:f\in B_1^\mathbb{C}(V[-a,a])\}.$$
Just as above,
we get a compact metric flow $(B_1V_{-a}^a,\mathbb{R})$,
the $\mathbb{R}$-shift on $B_1V_{-a}^a$.
We consider these flows as
compact metric subflows of the Bebutov flow $(C(\mathbb{R},[0,1]),\mathbb{R})$, through the canonical embedding $f\mapsto \frac{f+1}{2}$, and we use them to construct the explicit universal flow in our main theorem:

\medskip

\begin{Theorem}[Main theorem]\label{mth}
Let $M=\prod_{n\in\mathbb{N}}B_1V_{-n}^n$.
Then the compact metric flow $(M,\mathbb{R})$ is a universal flow for all compact metric flows.
\end{Theorem}

Note that the compact metric universal space we provide
is a countable product of compact subspaces of
the Bebutov flow $(C(\mathbb{R},[0,1]),\mathbb{R})$,
and therefore explicitly embeds into $(C(\mathbb{R},[0,1]^\mathbb{N}),\mathbb{R})$.

\medskip

We remark here that there are some previously known results related to universal flows.
We will give a universal flow in Section 4 via
a $C^\ast$-algebraic approach due to de Vries.
Other constructions are possible.
However, this approach builds a universal space which does not explicitly
embed into $(C(\mathbb{R},[0,1]^\mathbb{N}),\mathbb{R})$.

\medskip

This paper is organized as follows.
In Section 2, we show that
any compact metric flow embeds into $(C(\mathbb{R},[0,1]^\mathbb{N}),\mathbb{R})$.
In Section 3, we show that
any compact subflow of $(C(\mathbb{R},[0,1]^\mathbb{N}),\mathbb{R})$ embeds into
$(\prod_{n\in\mathbb{N}}B_1V_{-n}^n,\mathbb{R})$, which essentially completes the proof of Theorem \ref{mth}.
Finally, Section 4 compares our explicit construction in the main result
with a previously known result
related to universal flows.
The paper ends with an open question.

\medskip

\textbf{Acknowledgement.}
Y.G. was partially supported by the Marie Curie grant PCIG12-GA-2012-334564 and by the National Science Center (Poland) grant 2016/22/E/ST1/00448. The authors were partially supported
by the National Science Center (Poland) grant 2013/08/A/ST1/00275.

\medskip

\section{A natural noncompact universal flow}

In this section,
we embed all compact metric flows into the countable product of
the Bebutov flow, namely, the flow $(C(\mathbb{R},[0,1]^\mathbb{N}),\mathbb{R})$.
Note that $C(\mathbb{R},[0,1]^\mathbb{N})$ exactly consists of all those functions
$f=(f_i)_{i\in\mathbb{N}}$ with $f_i\in C(\mathbb{R},[0,1])$ for every $i\in\mathbb{N}$.

\begin{Lemma}\label{lem1}
The flow $(C(\mathbb{R},[0,1]^\mathbb{N}),\mathbb{R})$ is a universal flow for all compact metric flows
$(X,\mathbb{R})$.
\end{Lemma}

\begin{proof}
Write $(C(\mathbb{R},[0,1]^\mathbb{N}),\mathbb{R})=(C(\mathbb{R},[0,1]^\mathbb{N}),(\psi_r)_{r\in\mathbb{R}})$,
and let $(X,\mathbb{R})=(X,(\varphi_r)_{r\in\mathbb{R}})$ be a compact metric flow.
Since $X$ is a compact metric space, there exists a homeomorphism
$\phi:X\rightarrow\phi(X)\subset[0,1]^\mathbb{N}$
(for details see \cite[Theorem 34.1]{M}).
Define $F:X\rightarrow C(\mathbb{R},[0,1]^\mathbb{N})$ as follows:
for every $x\in X$, $F(x)(t)=\phi(\varphi_t(x))$ for all $t\in\mathbb{R}$.
Since $\phi:X\rightarrow [0,1]^\mathbb{N}$ is continuous and $\varphi_t(x)$ is continuous with respect to $t\in\mathbb{R}$,
we know that for every $x\in X$, $F(x)\in C(\mathbb{R},[0,1]^\mathbb{N})$.

For any $x\in X$, let us write $\phi(x)=(\phi_n(x))_{n\in\mathbb{N}}\in[0,1]^\mathbb{N}$.
Denote by $d_X$ the metric on $X$ and by $d$ the metric on $(C(\mathbb{R},[0,1]^\mathbb{N}),\mathbb{R})$.
Since $X$ is compact and $\phi$ is continuous, $\phi$ is uniformly continuous on $X$.
Fix $x\in X$, $m\in\mathbb{N}$ and $N\in\mathbb{N}$.
Then for every $\epsilon>0$ there is $\xi>0$ such that whenever $z_1,z_2\in X$ with $d_X(z_1,z_2)<\xi$
we have $|\phi_n(z_1)-\phi_n(z_2)|<\epsilon$ for each $n\in\mathbb{N}$ with $n\le m$.
Since $X$ is compact,
for such $\xi>0$ there exists $\delta>0$ such that
for every $y\in X$ satisfying $d_X(y,x)<\delta$ we have $d_X(\varphi_t(y),\varphi_t(x))<\xi$ for all $t\in[-N,N]$,
which implies that
$$||F(y)_n-F(x)_n||_{L^\infty([-N,N])}=\sup_{t\in[-N,N]}|\phi_n(\varphi_t(y))-\phi_n(\varphi_t(x))|\le\epsilon$$
for all $n\in\mathbb{N}$ with $n\le m$.
Thus, for every $x\in X$ and every $\epsilon>0$ we can find $\delta>0$ such that for all
those $y\in X$ with $d_X(y,x)<\delta$
we have $d(F(y),F(x))<\epsilon$. This shows that
$F:X\rightarrow C(\mathbb{R},[0,1]^\mathbb{N})$ is continuous.

It is clear that $F:X\rightarrow F(X)$ is one-to-one
since $\varphi_t$ and $\phi$ are homeomorphisms.
Since $X$ is compact and
$F:X\rightarrow F(X)$ is continuous and one-to-one,
we get that the map $F:X\rightarrow F(X)$ is a homeomorphism.

For every $r\in\mathbb{R}$ and $x\in X$, we have
$$F(\varphi_r(x))(t)=\phi(\varphi_t(\varphi_r(x)))=\phi(\varphi_{t+r}(x))
=F(x)(t+r)=\psi_r(F(x))(t)$$
for any $t\in\mathbb{R}$,
which shows that $F$ is $\mathbb{R}$-equivariant:
$F\circ\varphi_r=\psi_r\circ F$ for all $r\in\mathbb{R}$.
Thus, $(X,\mathbb{R})$ embeds into $(C(\mathbb{R},[0,1]^\mathbb{N}),\mathbb{R})$.
\end{proof}

\medskip

Next we consider compact subflows of the flow
$(C(\mathbb{R},[0,1]^\mathbb{N}),\mathbb{R})$.
Let $(X,\mathbb{R})$ and $(Y,\mathbb{R})$ be flows.
We say that $(Y,\mathbb{R})$ is a {\it subflow} (or sometimes, a {\it subspace}) of $(X,\mathbb{R})$
if $Y$ is an $\mathbb{R}$-invariant subset of $X$ endowed with the subspace topology
and $\mathbb{R}$ acts as the restriction on $Y$.
In order to prove Theorem \ref{mth} it suffices to show that
$(M,\mathbb{R})$ is universal for all compact subflows of $(C(\mathbb{R},[0,1]^\mathbb{N}),\mathbb{R})$.

\medskip

\begin{Theorem}\label{thm1}
Let $M=\prod_{n\in\mathbb{N}}(B_1V_{-n}^n)^n$.
Then the compact metric flow $(M,\mathbb{R})$ is a universal flow
for all compact subflows of $(C(\mathbb{R},[0,1]^\mathbb{N}),\mathbb{R})$.
\end{Theorem}

Clearly, $\prod_{n\in\mathbb{N}}(B_1V_{-n}^n)^n$ is a subspace of $\prod_{n\in\mathbb{N}}B_1V_{-n}^n$.
For simplicity we write $$M=\prod_{n\in\mathbb{N}}B_1V_{-n}^n$$ in the statement of Theorem \ref{mth}.

\medskip

\begin{proof}[Proof of Theorem \ref{mth} assuming Theorem \ref{thm1}]
Let $(X,\mathbb{R})$ be a compact metric flow.
According to Lemma \ref{lem1}, the flow $(X,\mathbb{R})$ embeds into $(C(\mathbb{R},[0,1]^\mathbb{N}),\mathbb{R})$ via an embedding $f$.
Since $X$ is compact, $(f(X),\mathbb{R})$ is a compact subflow of $(C(\mathbb{R},[0,1]^\mathbb{N}),\mathbb{R})$.
Thus, by Theorem \ref{thm1},
$(f(X),\mathbb{R})$ embeds into $(M,\mathbb{R})$,
which means that
$(X,\mathbb{R})$ embeds into $(M,\mathbb{R})$.
This ends the proof.
\end{proof}

\medskip

\section{Proof of Theorem \ref{thm1}}

\begin{proof}[Proof of Theorem \ref{thm1}]

Take a continuous function $\psi:\mathbb{R}\rightarrow[0,1]$ as follows (a tent function):
$\psi(x)=0$ for $x\in(-\infty,-1]\cup[1,+\infty)$; $\psi(0)=1$;
and $\psi$ is linear on both $[-1,0]$ and $[0,1]$.
For every $n\in\mathbb{N}$, set $\psi_n(x)=\psi(x/n)$ for all $x\in\mathbb{R}$,
and let $\varphi_n=\overline{\mathcal{F}}(\psi_n)$.
Then we have the following:

\medskip

\noindent {\bf Claim.}
$(\varphi_n)_{n\in\mathbb{N}}$ is a sequence of functions over $\mathbb{R}$ satisfying:
for every $n\in\mathbb{N}$, $\varphi_n:\mathbb{R}\rightarrow\mathbb{R}$ is continuous,
$\varphi_n\ge0$, $\int_\mathbb{R}\varphi_n(t)dt=1$,
$\supp\mathcal{F}(\varphi_n)\subset[-n,n]$; and
for every $h\in C(\mathbb{R},[0,1])$ and every $N\in\mathbb{N}$,
it holds that
$||h\ast\varphi_n-h||_{L^\infty([-N,N])}\rightarrow0$
as $n\rightarrow\infty$,
where $h\ast\varphi_n$ denotes the convolution of $h$ and $\varphi_n$.
\footnote{Note that $(\varphi_n)_{n\in\mathbb{N}}$ is called a
positive summability kernel (see \cite[Chapters I.2, VII.2]{Ka}).}

\medskip

In fact,
put $\varphi=\overline{\mathcal{F}}(\psi)$,
then for each $x\in\mathbb{R}$ we have
\begin{align*}
\varphi(x)&=\overline{\mathcal{F}}(\psi)(x)
=\int_\mathbb{R}e^{2\pi\sqrt{-1}xt}\psi(t)dt\\
&=2\int_0^{+\infty}\psi(t)\cos(2\pi xt)dt
=2\int_0^1(1-t)\cos(2\pi xt)dt\\
&=\frac{2\sin(2\pi x)}{2\pi x}-\frac{2}{(2\pi x)^2}(2\pi x\sin(2\pi x)+\cos(2\pi x)-1)
=\frac{1-\cos(2\pi x)}{2\pi^2x^2}\\
&=\frac{\sin^2(\pi x)}{\pi^2x^2}.
\end{align*}
Thus, $\varphi:\mathbb{R}\rightarrow\mathbb{R}$ is continuous, $\varphi\ge0$,
and it holds that
\begin{align*}
\int_\mathbb{R}\varphi(x)dx&=\int_\mathbb{R}\frac{\sin^2(\pi x)}{\pi^2x^2}dx
=\frac{1}{\pi}\int_\mathbb{R}\frac{\sin^2t}{t^2}dt\\
&=-\frac{1}{\pi}\int_\mathbb{R}\sin^2td(\frac{1}{t})
=\frac{1}{\pi}\int_\mathbb{R}\frac{\sin(2t)}{t}dt\\
&=\frac{1}{\pi}\int_\mathbb{R}\frac{\sin t}{t}dt
=1.
\end{align*}
Hence
for every $n\in\mathbb{N}$,
by noting that it holds for all $x\in\mathbb{R}$
that
$$\varphi_n(x)=\overline{\mathcal{F}}(\psi_n)(x)
=n\overline{\mathcal{F}}(\psi)(nx)=n\varphi(nx),$$
we have that
$\varphi_n:\mathbb{R}\rightarrow\mathbb{R}$ is continuous, $\varphi_n\ge0$,
and $$\int_\mathbb{R}\varphi_n(x)dx
=\int_\mathbb{R}|\varphi_n(t)|dt
=\int_\mathbb{R}n\varphi(nx)dx=\int_\mathbb{R}\varphi(t)dt=1.$$
Take $h\in C(\mathbb{R},[0,1])$ and $N\in\mathbb{N}$.
Since for every $n\in\mathbb{N}$ we have $\int_\mathbb{R}\varphi_n(t)dt=1$,
it holds that
$$|h\ast\varphi_n(x)-h(x)|=\left|\int_\mathbb{R}(h(x-t)-h(x))\varphi_n(t)dt\right|$$
for all $x\in\mathbb{R}$.
Fix $\epsilon>0$.
Then we can choose $\delta\in(0,1)$ such that
$|h(y_1)-h(y_2)|<\epsilon/2$ whenever $y_1,y_2\in[-N-1,N+1]$ with $|y_1-y_2|<\delta$.
Since
$\int_\mathbb{R}|\varphi(y)|dy<\infty$,
there exists
$A>0$ satisfying $\int_{\mathbb{R}\setminus[-A,A]}|\varphi(y)|dy<\epsilon/4$.
Take $m\in\mathbb{N}$ with $m\delta>A$.
Then for any $n\in\mathbb{N}$ with $n>m$,
we have $n\delta>m\delta>A$ and hence
$$\int_{\mathbb{R}\setminus[-\delta,\delta]}|\varphi_n(t)|dt
=\int_{\mathbb{R}\setminus[-\delta,\delta]}n|\varphi(nt)|dt
=\int_{\mathbb{R}\setminus[-n\delta,n\delta]}|\varphi(y)|dy
\le\int_{\mathbb{R}\setminus[-A,A]}|\varphi(y)|dy
<\frac{\epsilon}{4}.$$
It follows that
\begin{align*}
|h\ast\varphi_n(x)-h(x)|
&\le\int_\mathbb{R}|h(x-t)-h(x)|\cdot|\varphi_n(t)|dt\\
&=\int_{\mathbb{R}\setminus[-\delta,\delta]}|h(x-t)-h(x)|\cdot|\varphi_n(t)|dt+\int_{-\delta}^\delta|h(x-t)-h(x)|\cdot|\varphi_n(t)|dt\\
&\le\int_{\mathbb{R}\setminus[-\delta,\delta]}2||h||_{L^\infty(\mathbb{R})}\cdot|\varphi_n(t)|dt
+\int_{-\delta}^\delta\frac{\epsilon}{2}\cdot|\varphi_n(t)|dt\\
&\le2\int_{\mathbb{R}\setminus[-\delta,\delta]}|\varphi_n(t)|dt
+\frac{\epsilon}{2}\cdot\int_\mathbb{R}|\varphi_n(t)|dt\\
&<2\cdot\frac{\epsilon}{4}+\frac{\epsilon}{2}=\epsilon
\end{align*}
for all $x\in[-N,N]$.
Thus, $$||h\ast\varphi_n-h||_{L^\infty([-N,N])}\rightarrow0$$ as $n\rightarrow\infty$.
Finally,
for every $n\in\mathbb{N}$ we have
$$\supp\mathcal{F}(\varphi_n)=\supp\mathcal{F}(\overline{\mathcal{F}}(\psi_n))=\supp\psi_n\subset[-n,n].$$
This proves the claim.

\medskip

Let $(X,\mathbb{R})$ be a compact subflow of $(C(\mathbb{R},[0,1]^\mathbb{N}),\mathbb{R})$.
We are going to embed $(X,\mathbb{R})$ into $(M,\mathbb{R})$.
For each $n\in\mathbb{N}$, we define $F_n:C(\mathbb{R},[0,1]^\mathbb{N})\rightarrow(B_1V_{-n}^n)^n$ as follows:
$$F_n(f)=F_n((f_i)_{i\in\mathbb{N}})=(f_i\ast\varphi_n)_{i\in\{1,\cdots,n\}},$$
where $f=(f_i)_{i\in\mathbb{N}}\in C(\mathbb{R},[0,1]^\mathbb{N})$.

We should check that $f_i\ast\varphi_n\in B_1V_{-n}^n$ for any $n\in\mathbb{N}$ and $f_i\in C(\mathbb{R},[0,1])$.
In fact, since $||f_i||_{L^\infty(\mathbb{R})}\le1$ and $\int_\mathbb{R}|\varphi_n(y)|dy=1$, we have
$$|f_i\ast\varphi_n(x)|=\left|\int_\mathbb{R}f_i(x-y)\varphi_n(y)dy\right|\le\int_\mathbb{R}|\varphi_n(y)|dy=1$$
for all $x\in\mathbb{R}$, which shows that $||f_i\ast\varphi_n||_{L^\infty(\mathbb{R})}\le1$.
Fix $x_0\in\mathbb{R}$ and $\epsilon>0$. Since $\int_\mathbb{R}|\varphi_n(y)|dy<\infty$,
we can choose $A>0$ sufficiently large such that $\int_{\mathbb{R}\setminus[-A,A]}|\varphi_n(y)|dy<\epsilon/4$.
Since $f_i$ is uniformly continuous on $[x_0-A-1,x_0+A+1]$,
there exists $\delta\in(0,1)$ such that $|f_i(x_1)-f_i(x_2)|<\epsilon/2$ for any $x_1,x_2\in[x_0-A-1,x_0+A+1]$ with $|x_1-x_2|<\delta$.
Hence for any $x\in(x_0-\delta,x_0+\delta)$, we have
\begin{align*}
&|f_i\ast\varphi_n(x)-f_i\ast\varphi_n(x_0)|\\
\le&\int_\mathbb{R}|f_i(x-y)-f_i(x_0-y)|\cdot|\varphi_n(y)|dy\\
\le&\int_{\mathbb{R}\setminus[-A,A]}2||f_i||_{L^\infty(\mathbb{R})}\cdot|\varphi_n(y)|dy
+\int_{-A}^A|f_i(x-y)-f_i(x_0-y)|\cdot|\varphi_n(y)|dy\\
\le&2\cdot\frac{\epsilon}{4}+\frac{\epsilon}{2}\cdot\int_\mathbb{R}|\varphi_n(y)|dy
=\epsilon.
\end{align*}
This shows that $f_i\ast\varphi_n$ is continuous.
It is clear that $f_i\ast\varphi_n$ is a real function.
By noting that $\mathcal{F}(f_i\ast\varphi_n)=\mathcal{F}(f_i)\cdot\mathcal{F}(\varphi_n)$
(see \cite[Chapters VI.2, VI.4]{Ka})
and the fact that $\supp\mathcal{F}(\varphi_n)\subset[-n,n]$,
we have $$\supp\mathcal{F}(f_i\ast\varphi_n)
\subset\supp\mathcal{F}(f_i)\cap\supp\mathcal{F}(\varphi_n)
\subset\supp\mathcal{F}(\varphi_n)\subset[-n,n].$$
Thus, $f_i\ast\varphi_n\in B_1V_{-n}^n$.

\medskip

For $f\in C(\mathbb{R},[0,1]^\mathbb{N})$, define
$F(f)=(F_n(f))_{n\in\mathbb{N}}\in\prod_{n\in\mathbb{N}}(B_1V_{-n}^n)^n$;
namely,
$$F((f_1,f_2,f_3,\cdots))=(f_1\ast\varphi_1;f_1\ast\varphi_2,f_2\ast\varphi_2;
f_1\ast\varphi_3,f_2\ast\varphi_3,f_3\ast\varphi_3;\cdots \cdots),$$
where $(f_1,f_2,f_3,\cdots)\in C(\mathbb{R},[0,1]^\mathbb{N})$.
Then we obtain a map
$F:C(\mathbb{R},[0,1]^\mathbb{N})\rightarrow M$.

Consider $F:X\rightarrow M$.
It suffices to show that the map $F$ is an embedding of the flow
$(X,\mathbb{R})$ into the flow $(M,\mathbb{R})$.

\medskip

Write $(X,\mathbb{R})=(X,(\phi_r)_{r\in\mathbb{R}})$ and $(M,\mathbb{R})=(M,(\tau_r)_{r\in\mathbb{R}})$.
Then for every $r\in\mathbb{R}$ and every $f=(f_i)_{i\in\mathbb{N}}\in X$,
we have the following:
$$F(\phi_r(f))=\big((\phi_r(f_i)\ast\varphi_n)_{1\le i\le n, i\in\mathbb{N}}\big)_{n\in\mathbb{N}},\,\;\,
\tau_r(F(f))=\big((\tau_r(f_i\ast\varphi_n))_{1\le i\le n, i\in\mathbb{N}}\big)_{n\in\mathbb{N}}.$$
Since it holds that
\begin{align*}
\phi_r(f_i)\ast\varphi_n(x)
&=\int_\mathbb{R}\phi_r(f_i)(x-t)\varphi_n(t)dt\\
&=\int_\mathbb{R}f_i(x+r-t)\varphi_n(t)dt\\
&=f_i\ast\varphi_n(x+r)\\
&=\tau_r(f_i\ast\varphi_n)(x)
\end{align*}
for all $x\in\mathbb{R}$,
we get that $F(\phi_r(f))=\tau_r(F(f))$.
This indicates that
$F:X\rightarrow M$ is $\mathbb{R}$-equivariant:
for any $r\in\mathbb{R}$ we have
$F\circ\phi_r=\tau_r\circ F$.

\medskip

It remains to check that
the map $F:X\rightarrow F(X)$ is a homeomorphism.

Take $f=(f_i)_{i\in\mathbb{N}}\in X$ and $n\in\mathbb{N}$.
Recall that the metrics on $X\subset C(\mathbb{R},[0,1]^\mathbb{N})$ and $B_1V_{-n}^n$
are denoted by $d$ and $\boldsymbol{d}$ respectively.
Fix $N\in\mathbb{N}$ and $i\in\mathbb{N}$. Then
for any $\epsilon>0$, we can find $A>0$ satisfying
$\int_{\mathbb{R}\setminus[-A,A]}|\varphi_n(t)|dt<\epsilon/4$
and $\delta>0$ such that for every $g=(g_i)_{i\in\mathbb{N}}\in X$
with $d(g,f)<\delta$ it holds that $|g_i(y)-f_i(y)|<\epsilon/2$ for all
$y\in[-A-N,A+N]$.
Thus, for each $x\in[-N,N]$ we have
\begin{align*}
|g_i\ast\varphi_n(x)-f_i\ast\varphi_n(x)|
\le&\int_\mathbb{R}|g_i(x-t)-f_i(x-t)|\cdot|\varphi_n(t)|dt\\
\le&\int_{\mathbb{R}\setminus[-A,A]}2\cdot|\varphi_n(t)|dt
+\int_{-A}^A|g_i(x-t)-f_i(x-t)|\cdot|\varphi_n(t)|dt\\
\le&2\cdot\frac{\epsilon}{4}+\frac{\epsilon}{2}\cdot\int_\mathbb{R}|\varphi_n(t)|dt
=\epsilon.
\end{align*}
Hence for every $\epsilon>0$ there exists $\delta>0$ such that for any
$g=(g_i)_{i\in\mathbb{N}}\in X$ with $d(g,f)<\delta$ it holds that
$\boldsymbol{d}(g_i\ast\varphi_n,f_i\ast\varphi_n)<\epsilon$.
This, together with the fact that
$F(X)$ is a subspace of $M=\prod_{n\in\mathbb{N}}(B_1V_{-n}^n)^n$ equipped with the product topology,
proves that $F:X\rightarrow F(X)$ is continuous.

To see that $F:X\rightarrow F(X)$ is one-to-one,
we take $f=(f_i)_{i\in\mathbb{N}}$ and $g=(g_i)_{i\in\mathbb{N}}$ in $X$
with $f\ne g$.
This implies that there is $i\in\mathbb{N}$
and some $N>0$ such that $||f_i-g_i||_{L^\infty([-N,N])}>0$.
Since we have by the Claim that
$||f_i\ast\varphi_n-f_i||_{L^\infty([-N,N])}\rightarrow0$
and $||g_i\ast\varphi_n-g_i||_{L^\infty([-N,N])}\rightarrow0$ as $n\rightarrow\infty$,
there exists $n\in\mathbb{N}$ such that $||f_i\ast\varphi_n-g_i\ast\varphi_n||_{L^\infty([-N,N])}>0$,
which implies that
$f_i\ast\varphi_n\neq g_i\ast\varphi_n$.
Therefore, $F$ is injective.

Since $X\subset C(\mathbb{R},[0,1]^\mathbb{N})$ is compact and
$F:X\rightarrow F(X)$ is continuous and one-to-one,
we obtain that
$F:X\rightarrow F(X)$ is a homeomorphism.
This completes the proof.
\end{proof}

\medskip

\section{Final remarks}

In this section, we discuss a previously known result
related to universal flows,
and compare it with our explicit construction.

\medskip

This is an abstract method, due to J. de Vries \cite[Theorem 2.10]{Vri},
which compactifies $(C(\mathbb{R},[0,1]^\mathbb{N}),\mathbb{R})$ using the theory of $C^*$-algebras.
The existence of a universal space follows.
Recall that $(C(\mathbb{R},[0,1]^\mathbb{N}),\mathbb{R})$ is a separable metric flow.
Denote $(X,\mathbb{R})=(C(\mathbb{R},[0,1]^\mathbb{N}),\mathbb{R})$.
Notice that $X$ is a completely regular Hausdorff space.
Denote by $C_u(X)$ the Banach algebra of all complex-valued bounded continuous
functions on $X$ with the supremum norm. By
a $C_1^*$-subalgebra of $C_u(X)$ we mean a closed subalgebra of $C_u(X)$
containing the constants and closed under complex conjugation.
For an $\mathbb{R}$-action on $X$ we can naturally consider the induced $\mathbb{R}$-action on $C_u(X)$ by
$(tf)(x)=f(tx)$ for any $t\in\mathbb{R}$, $f\in C_u(X)$, and $x\in X$. In general, this action is not continuous
on $\mathbb{R}\times C_u(X)$.
We denote by $UC(X)$ the set of all $\mathbb{R}$-uniformly continuous functions on $X$;
equivalently, $UC(X)$ exactly equals the set of all those functions $f\in C_u(X)$ satisfying that
the mapping $\mathbb{R}\times C_u(X)\ni(t,g)\mapsto tg\in C_u(X)$ is continuous at all
points of $\mathbb{R}\times\{f\}$.
One first verifies that the subset
$UC(X)$ of $C_u(X)$ is actually an $\mathbb{R}$-invariant $C_1^*$-subalgebra of $C_u(X)$.
Next we let $M_1'$ be the Gelfand representation of $UC(X)$
(see e.g., \cite[Chapter VIII.3]{Ka}).
By well known theorems (see e.g., \cite[Section VIII.3]{Ka},\cite[Chapter 43]{GRS}),
$M_1'$ is a compact Hausdorff space,
and there is a canonical
continuous mapping $\phi:X\to M_1'$.
To get that $\phi$ is an embedding,
we need to know that the algebra $UC(X)$
is large enough to separate points and closed subsets
of $X$ (see \cite[Proposition 1.2]{Vri} and \cite[Theorem 14.2.2]{Sem}).

However $M_1'$ may fail to be metrizable.
By a clever argument, J. de Vries
shows that there is a separable $\mathbb{R}$-invariant $C_1^*$-subalgebra of $UC(X)$
separating points and closed subsets of $X=C(\mathbb{R},[0,1]^\mathbb{N})$
(see \cite[Proposition 2.10]{Vri}), which implies that $C(\mathbb{R},[0,1]^\mathbb{N})$
equivariantly embeds into some compact metrizable flow $(M_1,\mathbb{R})$,
which is a factor of $M_1'$ (see \cite[Theorem 2.10]{Vri}).

Thus, $M_1$ is a desired compactification of $X$.
Notice that this method is nonconstructive:
the space $M_1$ is not explicit.
However,
the technique works
not only for $\mathbb{R}$-actions
but also for locally compact group actions.

\medskip

As the above approach
yields a compact metric (universal) space
$(M_1,\mathbb{R})$,
this space embeds into
$(C(\mathbb{R},[0,1]^\mathbb{N}),\mathbb{R})$ as well.
However, in contrast to $(M,\mathbb{R})$ which was explicitly constructed
as a subspace of $(C(\mathbb{R},[0,1]^\mathbb{N}),\mathbb{R})$,
an embedding of $(M_1,\mathbb{R})$ into
$(C(\mathbb{R},[0,1]^\mathbb{N}),\mathbb{R})$
is far from explicit.

\medskip

Finally, we pose a question
which is related to Theorem \ref{KB}.

\medskip

Let $\mathcal{C}$ be the family of all compact metric flows $(X,\mathbb{R})$
whose fixed point set $F(X,\mathbb{R})$ is homeomorphic to a subset of $\mathbb{R}$.

\begin{Question}
Can one construct an explicit
compact subflow $(P,\mathbb{R})$ of the Bebutov flow $(C(\mathbb{R},[0,1]),\mathbb{R})$
which is universal for $\mathcal{C}$?
\end{Question}

\medskip


\end{document}